\documentclass[a4paper,12pt]{article}
\usepackage{parskip}
\usepackage[normalem]{ulem}
\usepackage{amsfonts,amsmath,amssymb,xcolor,hyperref,amsthm,graphicx,comment,enumerate,stmaryrd,color}
\usepackage{tikz}
\usetikzlibrary{positioning}

\title{Discrete averaging for discrete time dynamical systems}

\author{V.~Gelfreich$^1$ and A.~Vieiro$^{2,3}$\\[4pt]
$^1$ \small Mathematics Institute, University of Warwick, \small Coventry CV4 7AL, UK\\
\small{\tt v.gelfreich@warwick.ac.uk}\\[4pt]
$^2$ \small Departament de Matem\`atiques i Inform\`atica,\\ \small Universitat de Barcelona, Gran Via 585, 08007 Barcelona, Spain\\
$^3$ \small Centre de Recerca Matem\`atica (CRM),\\ \small Campus Bellaterra, 08193 Bellaterra, Spain\\
\small{\tt vieiro@maia.ub.es}
}

\providecommand{\keywords}[1]
{
  \small	
  \textbf{\textit{Keywords:}} #1
}

\newcommand{\N}{\mathbb{N}}
\newcommand{\R}{\mathbb{R}}
\newcommand{\C}{\mathbb{C}}
\newcommand{\Z}{\mathbb{Z}}

\newcommand\id{\operatorname{id}}

\newtheorem{thm}{Theorem}
\newtheorem{lemma}[thm]{Lemma}
\newtheorem{remark}[thm]{Remark}

\begin{document}

\maketitle	
\begin{abstract}
 In this paper we develop the theory of discrete averaging  designed to study discrete time dynamical systems defined by iterates of a map. The discrete averaging uses weighted averages over a segment of trajectory to find an autonomous vector field that approximates the original map. 
The method provides a simple and effective tool for finding adiabatic invariants, both numerically and analytically.
It is capable of strengthening  various theorems of the classical averaging theory because it eliminates two intermediate steps used in the classical averaging: the suspension procedure that assigns a rapidly oscillating flow to the map and time-dependent coordinate changes that eliminate the dependence on time.  
We  discuss  two applications of the discrete averaging --- to  the dynamics of a near-identity map and to the dynamics of a map in a neighbourhood of a resonant fixed point.
We show that the discrete averaging provides explicit uniform   bounds
for approximation errors. We also show that the discrete averaging
can be used to establish domain of validity of adiabatic approximations
in numerical experiments.
\end{abstract}

\keywords{near-identity maps, discrete averaging, embedding a map into a flow}

\section{Introduction}

In Dynamical Systems there are two parallel and, to a large extent, equivalent perturbative theories: one concerns rapidly oscillating vector fields, while the other addresses near-identity maps. For flows, a powerful collection of analytic tools has been developed on the basis of averaging methods \cite{BM1961,Verhulst2023,Verhulst2026,SVM2007}. By contrast, numerical investigations and visualization of dynamics are often performed  in terms of maps. The transition from a flow to a map is achieved via a Poincaré section, whereas the inverse transition relies on a suspension construction that represents the map as the time-one map of a non-autonomous time-periodic vector field.

Averaging theory has numerous applications in the study of resonant phenomena across a broad range of disciplines, including classical mechanics—particularly in problems of stability and Arnold diffusion—plasma physics, and singular perturbation theory, among others. Its effectiveness frequently rests on the existence and properties of adiabatic invariants coming from the study of averaged systems. From a practical perspective, it is essential to obtain either analytic or numerical expressions for these invariants and to determine domains in which they provide accurate approximations of the dynamics.

In plasma physics, the averaging theory for near-identity maps and  the corresponding adiabatic invariants are routinely used for analysis of accelerator beam dynamics~\cite{DEF2004}.
For example, 
the paper  \cite{BFGH14} presents the analysis of beam splitting   that relies on an adiabatic invariant associated with a return map to control trapping into resonances. In this context, the conservative Hénon map serves as the simplest standard  model. Of particular interest is the bifurcation near the 1:4 resonance, which produces large trapping regions due to a degeneracy in the model \cite{Bir87}. In \cite{BFGH14}, the adiabatic invariant was computed through a transformation to a higher-order normal form. Using this study as a motivating example,
 we show that the discrete averaging procedure  allows one to determine the adiabatic invariant directly in the original coordinates, without recourse to normal form transformations, and to carry out a systematic analysis of the domain of validity of the adiabatic approximation.

The standard construction of an approximating vector field proceeds in two stages. First, a near-identity map is interpreted as the time-one map of a non-autonomous ordinary differential equation with rapidly oscillating time dependence. Second, successive near-identity coordinate changes (averaging transformations) are employed to eliminate the time dependence order by order, yielding an autonomous vector field whose time-one flow approximates the original map. An alternative direct approach is based on expansions in powers of a small parameter $\varepsilon$, combined with a sequence of transformations removing the time dependence order by order in $\varepsilon$. This semi-algebraic formal procedure is often referred to as the Takens process \cite{Tak74}.

In general, a near-to-identity map cannot be embedded exactly into a flow \cite{Neishtadt84}. Consequently, the procedures described above typically lead to divergent series for the approximating vector field. The accuracy of an optimal truncation depends both on the distance from the identity and on the regularity of the system. Neishtadt \cite{Neishtadt84} proved that an analytic family of maps tangent to the identity can be embedded into a family of autonomous flows up to an exponentially small error.

The classical averaging approach has intrinsic limitations, as it relies on coordinate transformations that are not always amenable to precise analytic control or efficient numerical implementation. In this paper, we develop an alternative method based on weighted averages of iterates of the map. It generalizes the method introduced in \cite{GelfreichV18,GelfreichV26}.

Discrete averaging is particularly advantageous in numerical simulations, since its implementation is straightforward and computationally efficient. For instance, in \cite{GelfreichV18}, interpolating vector fields were used to visualize the dynamics of four-dimensional symplectic maps via a discrete-time analogue of Poincaré sections, combined with effective computations of adiabatic invariants in the original coordinates. 

The discrete averaging procedure provides explicit expressions for the approximating vector fields and yields a direct mechanism for controlling the approximation error. 
Such estimates were instrumental in the proof of a discrete version of the Nekhoroshev theorem in \cite{GelfreichV26}.
 In this paper, we derive explicit bounds for  error terms in the approximation by the time-one flow map. As a consequence, we obtain a refined version of Neishtadt’s theorem with quantitative estimates. Unlike the original result, our assumptions do not require the map to be a member of a
tangent-to-identity family and the approximation error is explicitly controlled by the ratio $\delta/\epsilon$, where $\epsilon$ measures the distance to the identity in a complex $\delta$-neighbourhood of the domain. 

Finally, the explicit construction of the Hamiltonian associated with a symplectic tangent-to-identity family of maps provides a systematic framework for the analysis of bifurcations of fixed points. It suffices to iterate the jet of the map (with respect to the phase variables and parameters) in the original coordinates; discrete averaging then yields the corresponding jet of a Hamiltonian vector field. In particular, near a  fully resonant point, a suitable iterate of the map becomes near-identity, and the symmetries obtained via normal form reduction are reflected in the invariance properties of the Hamiltonian under iterates of the map. 

The paper is organized as follows: in Section~\ref{Se:discreteav} we describe the 
definition of an interpolating vector field used in discrete averaging
and compare its definition with the classical averaging. Then we apply 
the discrete averaging to the study of the degenerate 1:4 resonance in the conservative Hénon map. In Section~\ref{Se:families} we establish the decay rate for approximation
of a tangent to the identity family of maps by a family of interpolating flows. 
In Section~\ref{Se:neishtadt} we prove uniform bounds for the approximation error for analytic maps and deduce the refined version of Neishtadt's theorem. 
In Section~\ref{Se:fixedpoint} we apply discrete averaging to obtain 
approximation for analytic tangent to identity maps in a neighbourhood of the fixed point. 
Finally, in Section~\ref{Se:stability} to prove exponentially long stability times for a symplectic map near
a fully resonant fixed point.

\section{Discrete averaging\label{Se:discreteav}}
\subsection{Classical averaging vs discrete averaging}

The classical averaging is usually used to study a system of differential equation
having two different time scales. The difference in time scales can originate
from an explicit dependence on a parameter of from scaling of a part of the phases space.
Hence, the classical applications of the averaging theory include the pendulum with a rapidly oscillating 
suspension point
\[
\ddot x=(1+a\sin(t/\varepsilon))\sin x
\]
and the pendulum with slowly changing frequency
\[
\ddot x=\omega(\varepsilon t)\sin x.
\]
In these examples $\varepsilon$ is a small parameter that characterizes
the ratio of timescales.

In general the averaging theory studies systems that can be written in the form
\begin{equation}\label{Eq:slow-fast}
\dot x=a(x,\varphi),\qquad \varepsilon \dot \varphi =b(x,\varphi),
\end{equation}
where $a$ and $b$ are periodic in $\varphi$ with a period $T>0$.
If $b\ge c>0$, the implicit function theorem (IFT) can be used to exclude the rapid phase $\varphi$ from the equation
and reduce the study to a non-autonomous equation of the form
\[
x'=\varepsilon f(x,\tau)
\]
with a non-autonomous time-periodic  vector field $f$. The solution of the averaged equation 
\[
y'=\varepsilon\bar f(y),\qquad\text{where}\quad \bar f(y)=\frac{1}{T}\int_0^Tf(y,\tau)d\tau,
\]
gives the leading order approximation for  trajectories
starting from a point $x_0$ at $\tau=0$: if  $x(0,x_0)=y(0,x_0)=x_0$ then
\[
x(\tau,x_0)=y(\tau,x_0)+O(\varepsilon)\qquad\text{for $|\tau|\le C/\varepsilon$}.
\]
The averaged equation is autonomous. Its integrals of motion play the role of
adiabatic invariants for the original system as they change slowly in time. 
The analysis of adiabatic invariants often requires development of
higher order adiabatic theory: a close to the identity time-periodic change of coordinates, $z=u_n(x,\tau)$, can
be used to eliminate the time-dependence in the leading orders: 
\[
z'=\varepsilon f_n(z,\varepsilon)+\varepsilon^{n+1} g(z,\tau,\varepsilon).
\]
An integral $h_n$ of the truncated system
\[z'=\varepsilon f_n(z,\varepsilon)\]
evaluated on a trajectory of the original system oscillates with an amplitude $O(\varepsilon)$
but stays near its initial value for substantially longer times $O(\varepsilon^{-n})$.
Indeed, suppose that a function $h_n$ is constant on trajectories of the truncated system. 
Then
the chain rule implies that the time derivative on the solution of the full
system $(h_n(z(\tau)))'=O(\varepsilon^{n+1})$. Consequently, 
$h_n$ stays $\varepsilon$-close to its initial value for times up to $O(\varepsilon^{-n})$. Since $u_n(x,\tau)=x+O(\varepsilon)$,
$h_n(x(\tau))$ oscillates around its initial value on the same timescale.
It is a remarkable property as the trajectories of the original and averaged systems
do not necessary remain close to each other on the timescale beyond $O(\varepsilon^{-1})$. 
This procedure provides a firm theoretical framework for the analysis of dynamics. On the other hand, computation of an adiabatic invariant, especially of higher orders, remains a challenging task as it relies on finding suitable coordinate changes. 

The discrete averaging suggests an alternative approach to computation of adiabatic invariants.
Suppose that $(x(t),\varphi(t))$ is a solution of the system \eqref{Eq:slow-fast}. Since $\varphi(t)$ is monotone increasing there is an increasing sequence $t_k$ that correspond to consecutive intersections of the solution
with the Poincaré section $\varphi=0\pmod{T}$. Then $x_k=x(t_k)$
is a trajectory of the first return map $F_\varepsilon:x_{k}\mapsto x_{k+1}$.
Since the return time is small, the first return map  is close to the identity, $F_\varepsilon=\operatorname{id}+O(\varepsilon)$.
The classical averaging provides an
approximation of $F_\varepsilon$ by a time-one map
of an autonomous vector field. 
In contrast to the classical averaging, 
the discrete averaging provides explicit expressions for interpolating vector fields $X_{n,\varepsilon}$ such that $F_\varepsilon=\Phi^1_{X_{n,\varepsilon}}+O(\varepsilon^{n+1})$,
the expressions are  given 
in terms of weighted averages over a trajectory segment.
Then the adiabatic invariants can be obtained from the analysis of the autonomous
vector fields $X_{n,\varepsilon}$.

\subsection{Interpolating vector fields}

The construction of an interpolating vector field is based on a 
finite segment of a trajectory of a map $F$. This map may depend
on one or more parameters but since all parameters are fixed during
the construction we can hide this dependence without risk of confusion.
Of course when $F$ depends on a parameter,  the corresponding interpolating vector field also does.

Let $n_0$ and $n$ be two integers with $0\le n_0\le n$. The integer $n$
defines the length of a trajectory segment and $n_0$ is used to select
a starting point. 
There is a unique $P_n(t,x)$, 
a polynomial of degree $n$ in $t$,
which interpolates a segment of the orbit of a point~$x$:
\begin{equation}\label{Eq:interpol}
P_n(k,x)=F^k(x)\quad\mbox{for $k=-n_0,\ldots,n-n_0$}.
\end{equation}
An interpolating vector field is the derivative of the interpolating polynomial at $t=0$:
\begin{equation}\label{Eq:Xn}
X_{n}(x)=\frac{\partial P_n}{\partial t}(0,x).
\end{equation}
We call our method the {\em discrete averaging} as the vector field $X_n$
is a weighted sum of the iterates of $F$. Indeed, 
the standard Lagrange interpolation procedure 
gives an explicit expression for the interpolating polynomial
\begin{equation}\label{Eq:LagrangeInterpol}
   P_n(t,x)=\sum_{k=-n_0}^{n-n_0}b_k(t)F^k(x) 
\end{equation}
where 
$b_k$ are the Lagrange basis polynomials,
\[
b_k(t)=\prod_{\substack {-n_0\le j \le n-n_0\\ j\ne k}}\frac{t-j}{k-j}.
\]
Consequently,
\begin{equation}\label{Eq:Xn_averaging}
X_{n}(x)=\sum_{k=-n_0}^{n-n_0}p_{nk}F^k(x)
,
\end{equation}
where the coefficients $p_{nk}=b_k'(0)$ depend on the choice of $n$ and $n_0$ but are independent of   the map $F$.
For example, in the cases $(n_0,n)=(0,1)$ and $(1,2)$ we get respectively
\begin{equation}\label{Eq:examples1-2}
X_1(x)=F(x)-x,\qquad X_2(x)=\tfrac12\bigl(F(x)-F^{-1}(x)\bigr).
\end{equation}
Any interpolating scheme
can be used to compute an interpolating vector field.
In particular, the Newton forward interpolation
scheme is a convenient tool for analytical proofs as it avoids inversion of the original map. 
On the other hand,  the Stirling-Newton central interpolation scheme
is often a 
better tool for numerical simulations as it provides more accurate approximations. 
For example, a center interpolation scheme
was used in \cite{GelfreichV18} was used for visualization of dynamics, while the forward interpolation was used
in the proof of the Nekhoroshev theorem~\cite{GelfreichV26}.

\subsection{An example: the Hénon map near a resonance\label{Se:henon}}

Before proceeding to the analysis of the discrete averaging
we discuss an application of the method to 
 the conservative Hénon map written  in the form similar to~\cite{SimVie09}:
\begin{equation} \label{henon}
H_c:(x,y)\mapsto  (c(1-(x+1)^2)+2x+y,-x).
\end{equation}
The linear reversing symmetry $R:(x,y)\mapsto (-y,-x)$
conjugates the Hénon map to its inverse:  $H_c^{-1} = R \circ H_c \circ R$. In particular we see that the inverse
map $H_c^{-1}$ is also polynomial. 

A straightforward computation shows that the map has two fixed points at $p_e=(0,0)$
and $p_h=(-2,2)$ respectively.
The fixed point $p_e=(0,0)$ is elliptic for $0<c<2$.
When $c=1$, the point $p_e$
is at the  1:4 resonance as the eigenvalues of $H_c'(p_e)$
are equal to $\pm i$.
In \cite{Bir87} it was shown that this resonance is not generic:
for $c=1+\varepsilon$ with $\varepsilon>0$ there is a 4-periodic hyperbolic orbit $p_4^h$ 
 at a distance $\mathcal{O}(\sqrt{\varepsilon})$ from $p_e$
 whose
invariant  manifolds form a chain of four islands of stability containing a
4-periodic elliptic orbit $p_4^e$ at a distance
$\mathcal{O}(\sqrt[4]{\varepsilon})$ from $p_e$. 
This is in contrast to the case of a non-degenerate stable 1:4
resonance where both elliptic and hyperbolic 4-periodic points are
located at a distance $\mathcal{O}(\sqrt{\varepsilon})$
form the elliptic fixed point.

A standard approach to analysis of the bifurcation taking place at $c=1$ is based on scaling of a $\varepsilon$-dependent neighbourhood of $p_e$ and considering the limit vector field  
of $H_{1+\varepsilon}^4$ as $\varepsilon\to0$. 
In the case of the Hénon map this approach fails to provide
the full picture of the bifurcation: 
choosing a scaling  that places $p_4^h$ at a unit distance
from $p_e$ sends $p_4^e$ to infinity, while choosing a scaling  that
places $p_4^e$
 at a unit distance from the origin forces the points of $p_4^h$ to collide at the origin. 
Thus, a more accurate approximation 
is needed  for analysis of the bifurcation.

We apply discrete averaging  to  the fourth iterate of 
the Hénon map  to obtain a suitable approximation.  
It suffices to use the second order central
interpolation scheme  \eqref{Eq:examples1-2} which yields
the interpolating vector field
\[
\begin{split}
X_2(x,y,\varepsilon)&=\tfrac12\left(H^4_{1+\varepsilon}(x,y)- H^{-4}_{1+\varepsilon}(x,y)\right).
\end{split}
\]
Since both $H_{1+\varepsilon}$ and $H_{1+\varepsilon}^{-1}$
are polynomial in $x,y$ and $\varepsilon$, the interpolating
vector field $X_2$ is also polynomial (of degree 16 in $(x,y)$
and of degree 31 in $(x,y,\varepsilon)$). 
Using an algebraic manipulator we check that
\[
H_{1+\varepsilon}^4-\Phi_{X_2}^1=O_{9}(x,y)+\varepsilon O_7(x,y)+O(\varepsilon^2)
\] 
where $O_k$ denotes a polynomial in $(x,y)$
that does not contain any term of order lower than $k$.
We note that the Hénon map is area-preserving but the vector field $X_2$ is only approximately Hamiltonian: 
its divergence is small but not vanishing,
\[
\operatorname{div }X_2(x,y,\varepsilon)= -12x^2y^2(x^3+y^3)+
O_8(x,y) 
 + 
\varepsilon O_4(x,y)
+O(\varepsilon^2)
.
\]
Since the divergence is small, the interpolating vector field $X_2$ is close
to a Hamiltonian vector field. We use integration to
extract the Hamiltonian part of the vector field. Let
\[
h_2(x,y,\varepsilon)=\int_0^1 X_2(t x,t y)\cdot(y,-x)dt 
\]
where $\cdot$ stands for the standard dot product. The function $ h_2$ is also polynomial and the direct comparison of the explicit expressions for $X_2$ and $ h_2$ 
gives
\[
X_2(x,y,\varepsilon)-\left(\frac{\partial  h_2}{\partial y}
,-\frac{\partial  h_2}{\partial x}\right) (x,y,\varepsilon)
=
O_8(x,y)+\varepsilon O_5(x,y)+O(\varepsilon^2).
\]
Computing the difference with the help of an algebraic manipulator we get that
\[
 h_2\circ  H_{1+\varepsilon}(x,y,\varepsilon)- h_2(x,y,\varepsilon)=O_9(x,y)+\varepsilon O_7(x,y)+O(\varepsilon^2).
\]
Consequently $ h_2$ is an approximate integral of motion 
for the Hénon map in a neighbourhood of the elliptic fixed point.
It is a remarkable fact as the construction of $ h_2$
is based on $H_{1+\varepsilon}^4$, the fourth iterate of the Hénon map,
while we observe conservation under the first iterate.
We provide an explanation in Section~\ref{Se:symmetries}.
Expanding $h_2$ in Taylor series of order 7 in $(x,y,\varepsilon)$ and dropping the terms of order $O(\varepsilon^2)$ we get
 a simpler polynomial function
 that provides an approximation for $h_2$ in a
neighbourhood of the origin:
\[
\begin{split}
\tilde h_2&=
- x^2 y^2+ x^4 y-x y^4-\frac{x^6}{3}+2 x^3 y^3-\frac{y^6}{3}+
  x^2 y^5-x^5 y^2
   \\&\quad
+\varepsilon\Bigl(
2  \left(x^2+y^2\right)+2  x y (y-x)+
\left(x^2-y^2\right)^2+ 3   x^4 y-2 x^3 y^2+2 x^2 y^3-3 x y^4
\\
&\qquad\qquad+
\tfrac{1}{3} \left(-4 x^6+6 x^5 y-17 x^4 y^2+24 x^3 y^3-17 x^2 y^4+6 x y^5-4 y^6\right)
\Bigr).
\end{split}
\]
Computing the difference with the help of an algebraic manipulator we get that
\[
 \tilde h_2\circ  H_{1+\varepsilon}(x,y,\varepsilon) - 
 \tilde h_2(x,y,\varepsilon)=O_9(x,y)+\varepsilon O_7(x,y)+O(\varepsilon^2).
\]
We see that $\tilde h_2$ is preserved up to an error of the same order as the error for the full polynomial $ h_2$. Figure~\ref{Fig:levelsH} shows the excellent agreement between level lines
of $\tilde h_2$ and plots of trajectories of the Hénon map $H_c$.

\begin{figure}
\begin{center}
\includegraphics[width=0.5\textwidth]{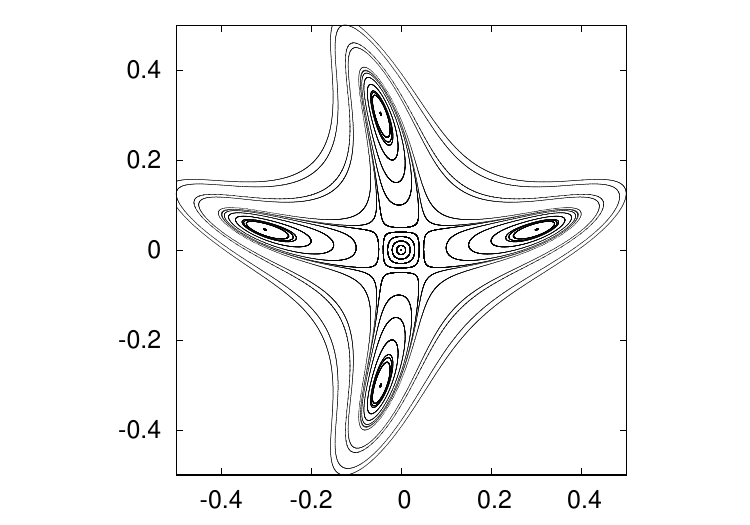}\includegraphics[width=0.5\textwidth]{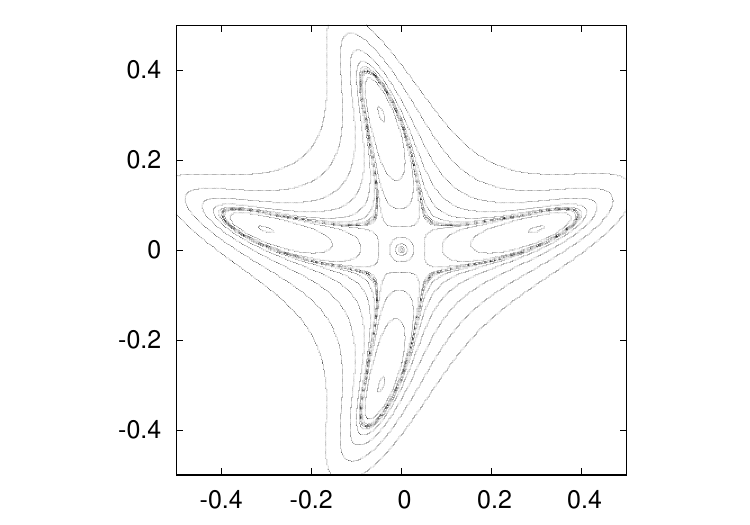}
\caption{Iterates of the Henon map \eqref{henon} (left) and 
level lines  of the Taylor polynomial of the adiabatic invariant $h_2$ of order 6 in $x,y$ and one in $\varepsilon$ (right) for $c=1+\varepsilon$ with  $\varepsilon=10^{-3}$.}
\label{Fig:levelsH}
\end{center}
\end{figure}

This example shows that the symmetric interpolating vector fields
of order two can be used to construct
an adiabatic invariant that is well preserved in a neighbourhood of the elliptic fixed point 
and correctly describes the degenerate bifurcation in the Hénon map. A higher accuracy approximation can be achieved by using
a higher order interpolation, for example, the fourth order
symmetric interpolating vector field defined by
\[
X_4(x,y,\varepsilon)=\frac{2}{3}\left(H_c^4(x,y,\varepsilon)
-H_c^{-4}(x,y,\varepsilon)\right)
-\frac1{12}
\left(H_c^8(x,y,\varepsilon)
-H_c^{-8}(x,y,\varepsilon)\right).
\]
The  computations  can be repeated with $X_4$ replacing $X_2$, and all errors move to a higher order. 
The adiabatic invariant $h_4$ obtained in this computation
 has the same lower orders as $h_2$ except for  two terms
in the last line of the expression for $\tilde h_2$
where we get $-13$ instead of $-17$. Of course, $h_4$
is preserved with a substantially higher accuracy.

Later in the paper we will show that 
using an interpolating vector field  of a higher order instead of $X_2$ it is possible to derive expressions for adiabatic invariants preserved up to an error of  arbitrarily high order
in $(x,y,\varepsilon)$. 
We note that the computation of
a higher order interpolating vector field requires the knowledge of iterates
of the map $H^{4k}_c$. These maps are polynomial but the degree
 grows quickly with $k$. Since terms 
of higher orders do not contribute to lower orders of the
adiabatic invariant, the symbolic computation can be substantially simplified by iterating an $n$-jet of the map,
instead of the full map, dropping the terms of order $n+1$
and higher at each iterate. 

Figure~\ref{Fig:levelsH}
shows that the adiabatic invariant is well preserved  for relatively large values of $x$ and $y$.
To determine  a domain $D$ where the discrete averaging can be used to analyze the dynamics we proceed as follows.
We consider the 
square $K=[-1,1]^2$. On a mesh of points $(x,y)$ in
$K$, we compute the integer $n$ minimizing $G_{(x,y)}(n)=|X_{2n}(x,y) -
X_{2n+2}(x,y)|_2$ where $X_{2n}$ is the symmetric
interpolating vector field of order $2n$.
Theorem~\ref{Thm:alaNeishtadt} implies that $G_{(x,y)}(n)=
\mathcal{O}(\epsilon^{2n+1})$  where $\epsilon$
characterizes  the closeness of the map to the identity.
Accordingly, for points in $D$  we expect that $G_{(x,y)}(n)$ first decreases as $n$ increases and then starts increasing.  The corresponding minimal value can be used as an indicator 
for the best accuracy that can achieved by the interpolating method. 
Outside $D$ the averaging 
approximation does not hold and we expect that $n=1$ minimizes $G_{(x,y)}(n)$.
The change in this behavior is expected to be quite abrupt, allowing us to
roughly determine the domain $D$ where discrete averaging can be used to approximate the dynamics. 
The Hamiltonian $h_2$ based on the second order interpolation scheme
defines an adiabatic invariant throughout $D$.
However, outside a neighborhood of the boundary of $ D$
 we can construct adiabatic invariants that are preserved with substantially higher accuracy.

Figure~\ref{nopt} illustrates an application of this analysis to the Hénon map
for two values of the parameter $\varepsilon$, before and after the bifurcation respectively. The top raw presents values of the optimal interpolation order $n$
for each  $(x,y)$ on the grid. The black color corresponds to the region
where $n=1$ and the discrete averaging does not provide a good approximation
for the dynamics. The white region corresponds to the area where $n\ge10$ and the
approximation is excellent. The 
second raw shows values of $G_{(x,y)}(n)$ for the optimal $n$. 
In the middle of the white region 
the error of a high order interpolation is below the computational precision
and for that reason we did not test methods above order 10.
We see that the boundary of the domain is indeed reasonably well defined.
The domain is rather large and fully covers the zone that contains the islands of
stability created in the bifurcation. 

\begin{figure}
\begin{center}
\includegraphics[width=0.5\textwidth]{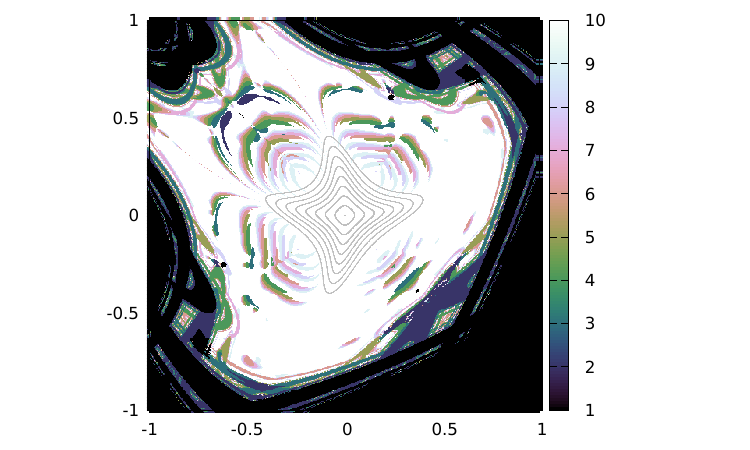}\includegraphics[width=0.5\textwidth]{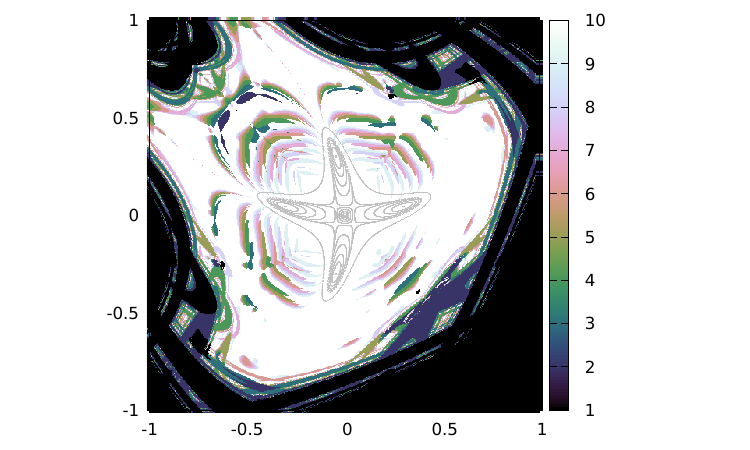}\\
\includegraphics[width=0.5\textwidth]{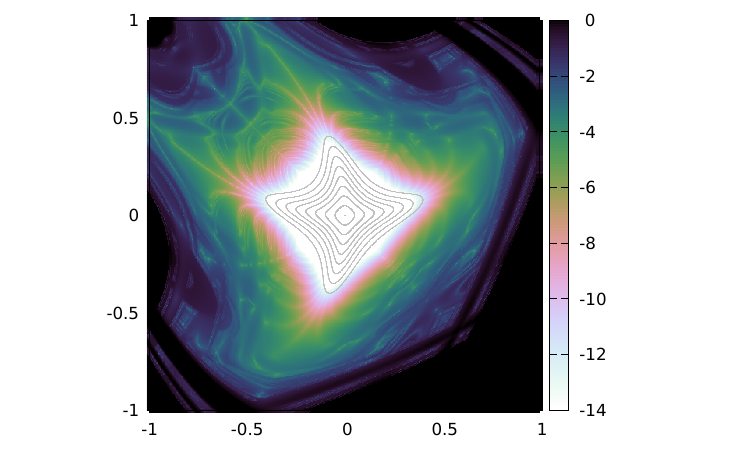}\includegraphics[width=0.5\textwidth]{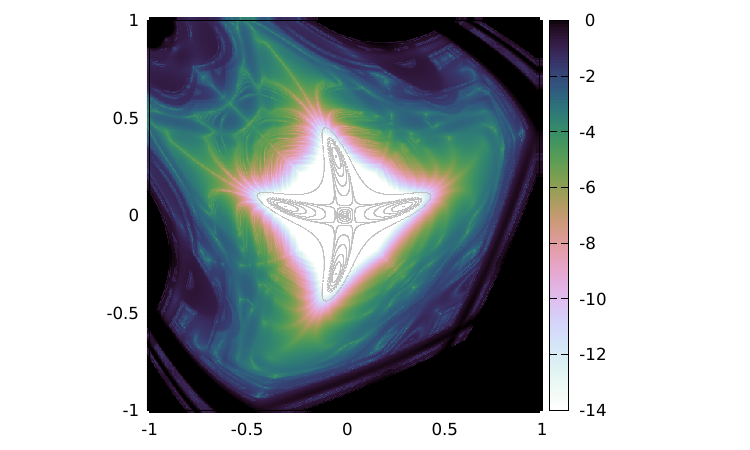}
\end{center}
\caption{Domain of validity of discrete averaging for the Hénon map with $\varepsilon=-10^{-3}$ (left column) and $\varepsilon=10^{-3}$ (right column). 
In the top row, the colors indicate the optimal $n$ values that minimize  $G_{(x,y)}(n)= |X_{2n}(x,y) - X_{2n+2}(x,y)|_2$.
The bottom row plots exhibit colors representing the 
decimal logarithm  of the minimum values of $G_{x,y}(n)$ over $n$. For reference, orbits of the Hénon map are shown in light gray in all plots.\label{nopt}}
\end{figure}

We emphasize that our computations  of the adiabatic invariant are performed in the original coordinates and do not use transformations to a normal form.
In order to accelerate computation of the symbolic expressions for
adiabatic invariants we use jet iteration.
On the other hand, the determination of  numerical values
for adiabatic invariants can be 
carried out from
numerically computed trajectories of the map,
similarly to  \cite{GelfreichV18}.
This approach permits analysis of bifurcations (including degenerate cases), determine its stability and providing an accurate adiabatic
invariant via direct computations in the original coordinates of the map, even when an analytic explicit closed form expression of the map is not available.

\goodbreak

\section{Approximation by autonomous flows\label{Se:families}}

Let $D\subset \R^d$ be an open domain and consider  a smooth family of maps $F_\varepsilon:D\to\R^d$
that is {\em tangent to identity}, i.e., 
\begin{equation}\label{eq:main_map}
F_\varepsilon(x)=x+\varepsilon f_\varepsilon(x)
\end{equation}
where
the Taylor expansion of 
$f_\varepsilon$ 
with respect to $\varepsilon$ up to an order $n$ 
is given by
\[
f_\varepsilon(x)=\sum_{k=0}^{n-1} \varepsilon^k f_k(x)+\varepsilon^n r_n(x,\varepsilon)
\]
with a bounded remainder $r_n$. It is well known that $F_\varepsilon$
can be approximated by an autonomous flow. The approximation error
of the best approximation depends on the smoothness of the family. 
The following theorem shows that the approximating flows
can be obtained using discrete averaging. 

\begin{thm} \label{Thm:formal}
If a tangent to identity family $F_\varepsilon\in C^{n+1}\bigl(D\times [-\varepsilon_0,\varepsilon_0]\bigr)$
for some $\varepsilon_0>0$,
then there is a unique polynomial in $\varepsilon$ vector field
\(
g_\varepsilon(x)=\sum_{k=0}^{n-1}\varepsilon^k g_k(x)
\)
such that 
\[
F_\varepsilon=\Phi^1_{\varepsilon g_\varepsilon}+O(\varepsilon^{n+1})
\]
uniformly on every compact subset of $D$. Moreover, for   $1\le k\le m \le n$
\begin{equation}\label{Eq:g_k}
\frac{1}{k!}\left.\frac{\partial^k X_m}{\partial \varepsilon^{k}}\right|_{\varepsilon=0}=g_{k-1},    
\end{equation}
where $X_m$ is an interpolating vector field of order $m$ defined by \eqref{Eq:Xn_averaging},
and consequently
\begin{equation}\label{Eq:errorbound}
F_\varepsilon=\Phi^1_{X_m}+O(\varepsilon^{m+1}).
\end{equation}
\end{thm}

\begin{proof}
The time-one flow defined by the ordinary differential equation (ODE)
\begin{equation}\label{eq:ODE}
\dot x=\varepsilon  g_\varepsilon(x)
\end{equation}
can be written in the form
\(\Phi_{\varepsilon g_\varepsilon}^t(x)=x+y(t,x),
\) 
where $\partial _ty(t,x)=\varepsilon g_\varepsilon(x+y(t,x))$
and $y(0,x)=0$. We suppress the explicit dependence of the function $y$ on the parameter $\varepsilon$.
Then 
\[
\partial_t^2y(t,x)=\varepsilon\partial_t  g_\varepsilon(x+y(t,x))=\varepsilon^2 g_\varepsilon'(x+y(t,x))g_\varepsilon(x+y(t,x))
=\varepsilon^2 D_{g_\varepsilon }g_\varepsilon (x+y(t,x)) 
\]
where $D_g$ is the directional derivative:
\[
D_gf(z)=f'(z)g(z)=\lim_{\delta\to0} \frac{f(z+\delta g(z))-f(z)}{\delta}.
\]
Then  for all $k$
\[
\partial_t^k y(t,x)
=\varepsilon^k D_{g_\varepsilon }^{k-1}g_\varepsilon (x+y(t,x)) .
\]
By expanding the solutions of the ODE into Taylor series in time
and evaluating time derivatives at $t=0$,
we arrive to an expression for the flow defined by the ODE~\eqref{eq:ODE}:
\begin{equation}\label{Eq:Phi_t}
\Phi^t_{\varepsilon g_\varepsilon}(x)=x+t\varepsilon g_\varepsilon(x)+\sum_{k=2}^{n} 
\frac{\varepsilon^k t^k}{k!}D_{g_\varepsilon}^{k-1}g_\varepsilon(x)+\varepsilon^{n+1}t^{n+1}q_n(t,x,\varepsilon)
.
\end{equation}
The Taylor expansions of $F_\varepsilon$ and $\Phi^1_{\varepsilon g_\varepsilon}$ coincide up to the order of $\varepsilon^n$
if and only if
\[
g_0=f_0\quad\text{and}\quad g_k=f_k-\sum_{j=2}^{k+1}\frac{1}{j!}\left[D_{g_\varepsilon}^{j-1}g_\varepsilon\right]_{k+1-j}
\quad\mbox{for $k=1,\ldots,n-1$,}
\]
where the operation $[\cdot]_j$ extracts the coefficient 
in front of $\varepsilon^j$ from a polynomial  in $\varepsilon$.
Using bi-linearity of the directional derivative it is easy to see that 
the sum in the right-hand side is a function of $g_0,\ldots,g_{k-1}$ and their derivatives. 
Consequently, the coefficients $g_k$ are defined uniquely and explicitly. 
After making this choice for the coefficients we obtain the vector field $g_\varepsilon$ such that
\[
F_\varepsilon(x)=\Phi^1_{\varepsilon g_\varepsilon}(x)+O(\varepsilon^{n+1}).
\]
Now consider the Lagrange interpolating polynomial $P_m$ defined by \eqref{Eq:LagrangeInterpol}.
We note that interpolation procedure is linear with respect to the interpolant 
and reproduces a polynomial interpolant exactly. 
Therefore, the equation \eqref{Eq:Phi_t} implies that 
\[
\sum_{k=-n_0}^{n-n_0}p_{nk}\Phi_{\varepsilon g_\varepsilon}^k(x)
=\varepsilon g_\varepsilon(x)+
\varepsilon^{n+1}
\sum_{k=-n_0}^{n-n_0}p_{nk}k^{n+1}q_n(k,x,\varepsilon)
\]
where the coefficients $p_{nk}$ are the same as in \eqref{Eq:Xn_averaging}.
Since $F_\varepsilon^k(x)=\Phi_{\varepsilon g_\varepsilon}^k(x)+O(\varepsilon^{n+1})$
for every fixed $k$,
we get
\[
X_{n}(x,\varepsilon)=\sum_{k=-n_0}^{n-n_0}p_{nk}F_\varepsilon^k(x)=
\sum_{k=-n_0}^{n-n_0}p_{nk}\Phi_{\varepsilon g_\varepsilon}^k(x)+O(\varepsilon^{n+1})
=\varepsilon g_\varepsilon(x)+O(\varepsilon^{n+1}).
\]
The smooth dependence of the flow on its vector field implies that
 the Taylor series of $F_\varepsilon$ and $\Phi^1_{X_n}$
coincide up to the order $O(\varepsilon^n)$. 
\end{proof}

\begin{remark}[Computing the formal interpolating flow using discrete averaging]
The equations~\eqref{Eq:Xn_averaging} and \eqref{Eq:g_k} imply that 
 coefficients of the formal interpolating vector field can be computed
using the forward interpolation scheme:
\[
g_{m-1}(x)=\sum_{k=1}^m\frac{p_{mk}}{m!}\left.\frac{\partial^m}{\partial \varepsilon^m}\right|_{\varepsilon=0}F^k_\varepsilon(x).
\]
The symbolic computation can be simplified by noting that 
in order to get an analytic expression for $g_{m-1}$ it is sufficient
to compute each of the  iterates  $F^k_\varepsilon$ keeping the terms up to the order $\varepsilon^m$
only.

This approach provides a new method for computation of the
formal interpolating vector field. On the other hand,
 $X_m$ has an important advantage over the truncated formal interpolator: 
its computation does not require knowledge of the derivatives of the  function
$F_\varepsilon$ and, consequently, it can be used 
 even in 
the case when no explicit analytic expression for $F_\varepsilon$ is  available.
\end{remark}

\section{Uniform bounds in the space of analytic diffeomorphisms\label{Se:neishtadt}}

In the previous section we established the rate of decay
for an approximation of a map $F_\varepsilon$ by an autonomous interpolating flow as the map approaches the identity when $\varepsilon$ goes to zero. If the family is infinitely smooth, then for any $m$ there is an interpolating vector field $X_m$
such that $F_\varepsilon=\Phi^1_{X_m}+O(\varepsilon^{m+1})$.
The error term decays faster for larger  $m$ but Theorem~\ref{Thm:formal} does not provide any information
on the constant in the $O$-term. Therefore it is not
possible to use this theorem to decide what order of interpolation leads to the most accurate approximation for a fixed value of $\varepsilon$.

In this section we answer this question 
for analytic diffeomorphisms. 
We show that all analytic maps from a fixed neighbourhood 
of the identity admit an approximation by
the time-one flow of interpolating vector fields with uniform bounds on the approximation error.
We prove that the approximation error is controlled by the ratio $\epsilon/\delta$ of two natural parameters:  $\epsilon$ describes the distance to the identity map
and $\delta$ characterizes the size of the complex neighbourhood where the map
is $\epsilon$-close to the identity. 

In contrast to Theorem~\ref{Thm:formal},
 the assumptions of the following theorem does not require
the map to be a member of a family, thus we drop notation of  potential dependence of $F$ on its parameters.

Suppose that $\epsilon$ and $\delta$ are two positive numbers. 
For a set $D_0\subset \C^d$, let  $D_\delta$ be a  $\delta$-neighbourhood 
of $D_0$. We  denote by $C^\omega(D_\delta,\C^d)$  the space of all
analytic maps $F:D_\delta\to \C^d$ equipped with 
the supremum norm
\[
\|F\|_{D_\delta}=\sup_{x\in D_\delta}|F(x)|,
\]
where we use the maximum component norm for vectors in $\C^d$.
Let $B_{\epsilon}(D_\delta)\subset C^\omega(D_\delta,\C^d)$ be the closed ball of radius $\epsilon$ centered
at the identity map:  an analytic map $F\in B_{\epsilon}(D_\delta)$ iff
\begin{equation}\label{Eq:f-xi}
\|F-\operatorname{id}\|_{D_\delta}\le \epsilon.
\end{equation}
If a segment of a trajectory $x_k=F^k(x)\in D_{\delta}$ for $k=0,\dots,m$,  
the Newton forward interpolation scheme provides an explicit 
expression for the interpolating vector field of order $m$ at the point $x$:
\begin{equation}\label{Eq:interpol_VF}
	X_m(x)=\sum_{k=1}^m\frac{(-1)^{k-1}}{k}\Delta_k(x),
\end{equation}
where the finite differences are defined recursively
\begin{equation}\label{Eq:finite-diff}
	\Delta_0(x)=x,
	\qquad 
	\Delta_{k}(x)=\Delta_{k-1}(F(x))-\Delta_{k-1}(x)
	\quad\text{for  $k\ge 1$}.
\end{equation}
The interpolating vector field of order one, $X_1(x)=F(x)-x$,
is similar to the vector field of the limit flow that is
traditionally used in the analysis of dynamics of a near-the-identity map.
The method also includes  higher order interpolation schemes that correspond to $m\ge2$.

Let $\Phi^t_X$ denote the flow defined by a differential equation $\dot x=X(x)$.
Let  integral $M_{\alpha}=\left\lfloor c_0/\alpha\right\rfloor+1
$ where the constant $c_0=1/(6\mathrm e)= 0.061\dots$.

\begin{thm}\label{Thm:alaNeishtadt}
Let $\epsilon,\delta>0$ be such that $\epsilon/\delta \le c_0$ and $D_0\subset \C^d$.
Then for any map $F\in B_{\epsilon}(D_\delta)$  the interpolating vector field 
of order one satisfies $\|X_1\|_{D_\delta}\le\epsilon $ and  
\begin{equation}
   \bigl\|\Phi^1_{X_1}-F \bigr\|_{D_0} 
  \le 2\epsilon^2/\delta.
\end{equation}
The interpolating vector field of order $m\in [2,M_{\epsilon/\delta}]$ satisfies $\|X_m\|_{D_{\delta/3}}
  \le 2\epsilon $ and
\begin{equation}\label{Eq:Phi-f}
    \bigl\|\Phi^1_{X_m}-F \bigr\|_{D_0} 
  \le   
    3\epsilon  \left( \frac{6 (m-1)\epsilon}{\delta}\right)^{m} .
\end{equation}
In particular, for $\displaystyle m =M_{\epsilon/\delta} $ 
\begin{equation}\label{Eq:expbound}
\bigl \|\Phi^1_{X_{\mathrm{m}}}-F \bigr\|_{D_0}\le 
3\,\epsilon \exp\left(- c_0\, \delta /\epsilon\right).
\end{equation}
\end{thm}

\begin{proof}
We consider the family that connects $F$ and the identity map,
\[
f_\mu=(1-\mu)\operatorname{id} +\mu F
\]
where $\mu$ is a complex parameter.  
Obviously the  function $f_\mu$ is analytic in the same domain $D_\delta$ as the function $F=f_1$.

\noindent{\bf 1. The limit flow.} First we consider the
 vector field 
\[
X_{1,\mu}=f_\mu -\operatorname{id} =\mu(F-\operatorname{id} ).
\] 
Equation \eqref{Eq:f-xi} implies
\[\|X_{1,\mu}\|_{D_{\delta}}=\|\mu(F-\operatorname{id} )\|_{D_{\delta}}\le |\mu|\epsilon.
\]
For $|\mu|\le\mu_1=\delta/\epsilon$, this norm does not exceed $\delta$, the distance from
$D_0$ to the boundary of $D_\delta$. Consequently,
the time-one flow $\Phi^1_{X_{1,\mu}}(x)$ is defined 
for all $x\in D_0$ and 
\[\bigl\|\Phi^1_{X_{1,\mu}}-\operatorname{id} \bigr\|_{D_0}
\le |\mu|\epsilon.
\]
Then the triangle inequality implies that
\[
\bigl\|\Phi^1_{X_{1,\mu}}-f_\mu \bigr\|_{D_0}\le 
\bigl\|\Phi^1_{X_{1,\mu}}-\operatorname{id} \bigr\|_{D_0}+\bigl\|\operatorname{id} -f_\mu \bigr\|_{D_0}
\le 2 |\mu|\epsilon. \]
 Theorem~\ref{Thm:formal} implies that $\Phi^1_{X_{1,\mu}}-f_\mu=O(\mu^2)$. The Maximum Modulus Principle (MMP\footnote{We  use the following simple statement of Complex Analysis.
 If a function $g$ is analytic and bounded in an open disk   $|\mu|<r$ and $g^{(k)}(0)=0$ for $k=0,1,\ldots,m-1$,
 then  the maximum modulus principle implies that 
 $|g(\mu )|\le \left(|\mu|/r\right)^{m} \sup_{|\mu|< r} |g(\mu )|$.
 Of course, if the function  extends continuously onto the boundary of the disk,
 the supremum can be replaced by the maximum over $|\mu|=r$. For a vector valued function
 we apply the MMP componentwise.})
 on the disk $|\mu|\le\mu_1$ for $\mu=1$ implies that
\[
\bigl\|\Phi^1_{X_{1}}-F \bigr\|_{D_0}  
\le   \frac{ 2\epsilon\mu_1}{\mu_1^{2}}
=
 \frac{2\epsilon^2 }{\delta}.
\]
{\bf 2. Bounds for the finite differences in $D_{\delta/3}$.}
Now we consider the contribution of higher order corrections to the interpolating
vector fields.
Let $m\ge 2$ and
\[
\mu_m=\frac{2\delta}{3\epsilon (m-1)}.
\]
Suppose $|\mu|\le \mu_m$, $x_0\in D_{\delta/3}$ and $0\le k\le m-1$.
The finite induction in $k$ implies that $x_k:=f_\mu^k(x_0)\in D_{\delta}$ 
and, consequently, $|x_{k+1}-x_k|\le |\mu|\epsilon $. 
Then the definition \eqref{Eq:finite-diff}
implies that 
\[
    \Delta_{m}(x_0)=\sum_{k=0}^{m-1} (-1)^{m-k-1} \binom{m-1}{k}\Delta_1(x_k).
\]
Since $\sum_{k=0}^{m-1}\binom{m-1}{k}=2^{m-1}$ and $|\Delta_1(x_k)|\le |\mu|\epsilon$, the triangle inequality gives us
\[
	\left\| \Delta_{m}\right\|_{D_{\delta/3}}
	\le  
	2^{m-1}|\mu|\epsilon\,.
\]
We note that  $\Delta_k(x)=\Delta_{k-1}(f_\mu(x))-\Delta_{k-1}(x)$
and, since $f_\mu(x)=x+O(\mu)$, the Taylor series in $\mu$ for $\Delta_k$ start
with a higher order than the Taylor series for $\Delta_{k-1}$.
We conclude that   $\Delta_{m}(x_0)=O(\mu^{m})$.
Applying the MMP we get 
\[
\left\|\Delta_{m}\right\|_{D_{\delta/3}}
\le
2^{m-1}\mu_{m}\epsilon \,
\left(\frac{|\mu|}{\mu_{m}}\right)^{m} \,. 
\]
This bound also holds for $m=1$.

{\bf 3. Bounds for $\bigl\|X_{m,\mu}\bigr\|_{D_{\delta/3}}$ with $2\le m\le M_{\epsilon/\delta}$.}
Let the interpolating vector field $X_{m,\mu}$ be defined 
by \eqref{Eq:interpol_VF} with $f$ replaced by $f_\mu$.
We immediately conclude that 
  $X_{m,\mu}$ is analytic in $D_{\delta/3}$
for $|\mu|\le \mu_{m}$. Since $\mu_k$ are decreasing, 
we get  the following upper bound 
for  $|\mu|\le \mu_m/4$:
\[
\begin{split}
 \bigl\|X_{m,\mu}\bigr\|_{D_{\delta/3}}
&
\le\sum_{k=1}^m \frac{ \left\|\Delta_k\right\|_{D_{\delta/3}}}{k}
\le
\epsilon\sum_{k=1}^m \frac
{2^{k-1}|\mu|^k}
{k\mu_{k}^{k-1}} 
\le
\epsilon\sum_{k=1}^m \frac
{2^{k-1}|\mu|^k}
{\mu_{m}^{k-1}} 
\le
\frac{\epsilon |\mu|}{1-\frac{2|\mu|}{\mu_m}}
\le 
2\epsilon |\mu|.
\end{split}
\]
Since $\mu_m>4$ for $2\le m\le M_{\epsilon/\delta}$  we can substitute $\mu=1$ and get
\[
\|X_{m}\|_{D_{\delta/3}}\le 2 \epsilon.
\]
{\bf 4. Approximation of $F$ by the time-one map of $X_{m}$ in $D_0$.}
Recalling the definition of $\mu_m$ we observe that the inequality 
\(
2\epsilon|\mu|
 \le\frac{\delta}{3(m-1)}\le\frac{\delta}{3}
\) follows from $|\mu|\le \mu_m/4$.
Consequently, an orbit of the vector field $X_{m,\mu}$ with an initial condition  in $D_0$ remains in $D_{\delta/3}$ during one unit of time and
$$
\bigl\|\Phi^1_{X_{m,\mu}}-\operatorname{id} \bigr\|_{D_0}\le \|X_{m,\mu}\|_{D_{\delta/3}}\le 2\epsilon|\mu|\,.
$$
Recalling that  $\|f_\mu - \operatorname{id} \|_{D_0} 
 \leq |\mu| \epsilon$ and using the triangle inequality we get
$$
\bigl\|\Phi^1_{X_{m,\mu}}-f_\mu \bigr\|_{D_0}\le \bigl\|\Phi^1_{X_{m,\mu}}-\operatorname{id} \bigr\|_{D_0}+\bigl\|\operatorname{id} -f_\mu \bigr\|_{D_0} \le 
3\epsilon|\mu|.
$$
Theorem~\ref{Thm:formal} implies that
$\Phi^1_{X_{m,\mu}}$ has the same Taylor polynomial of degree $m$ in $\mu$ as the map
 $f_\mu$. Then
the MMP  based on the bound in the disk 
$|\mu|\le \mu_m/4$ can be applied with $\mu=1$ to get the desired upper bound \eqref{Eq:Phi-f}:
\[
\bigl\|\Phi^1_{X_{m}}-F \bigr\|_{D_0}  
\le   3\epsilon \left( \frac{ 4}{\mu_{m}}\right)^{m}
=
 3\epsilon\,\left(  \frac{6\epsilon (m-1)}{\delta}\right)^{m} .
\]
In order to obtain the exponentially small approximation error we note that the right-hand side 
of the last inequality depends on $m$ and takes the least value 
 near  
$m=M_{{\epsilon/\delta}}=\left\lfloor \delta/(6 \mathrm e\epsilon)\right\rfloor+1$.
Since
 \(
\frac{\delta}{6\epsilon (M_{\epsilon/\delta}-1)}
\ge \mathrm e
\)
we conclude that for this $m$
$$
\bigl\|\Phi^1_{X_m}-F \bigr\|_{D_0} 
\le 3 \epsilon \, \mathrm e^{-M_\epsilon} 
\le 3  \,
\epsilon
\exp
\left(-
\frac{\delta}{6\mathrm e \epsilon}
\right).
$$
This argument completes the proof of the theorem.
\end{proof}

\begin{remark}[Symmetric interpolation schemes]
Theorem~\ref{Thm:alaNeishtadt} uses the forward interpolation scheme. A similar statement can be proved for any
other interpolating scheme provided the map and its inverse 
are both  $\epsilon$-close to the identity on the same domain.
The most interesting case is related to symmetric interpolation schemes as they potentially lead to more accurate approximations.  The analysis
of the inverse map is particularly simple in the class of reversible maps, see the motivating example in Section~\ref{Se:henon} and the paper \cite{GelfreichV18}
where the method was used to study Arnold diffusion
in symplectic maps.    
\end{remark}

\section{Application to tangent-to-identity maps
\label{Se:fixedpoint}}

Suppose that a local diffeomorphism $f$ has a fixed point
at the origin, $f(0)=0$. If additionally $f'(0)=I$,
we say that $f$ is  {\em tangent to the identity}. The
study of the corresponding dynamics in a complex neighbourhood
of the fixed point is a beautiful and complex subject,
non-trivial even in the one-dimensional case,
see for example \cite{Voronin1981,Ecalle1985,Gelfreich1998,Hakim1998,Abate2001}.
 It is well known that a tangent to the identity map can be formally embedded into an
 autonomous flow. However, in general, an autonomous flow can only provide an approximation up to an error exponentially small compared to the distance to the origin.
Theorem~\ref{Thm:alaNeishtadt} can be used to 
reproduce this statement. In this section we will derive sharper estimates  
taking into account  the  dynamics of the map in a complex neighbourhood of the origin.
The main contribution of our approach is
the uniformity of the  bounds.

Let us state the assumptions formally.
Consider the space of local analytic diffeomorphisms  $f$ 
such that the Taylor expansion of $f-\operatorname{id}$ 
does not have any term of an order less than ${n_0}\ge 2$,
i.e. $f(x)-x=O(x^{n_0})$  (or equivalently $\operatorname{val}(f-\operatorname{id})\ge n_0$).
\footnote{
The degree of the first non-zero monomial of the Taylor expansion in 
 $x$ defines a valuation, that will be denoted by $\operatorname{val}$, of the ring
 of formal series $\mathbb{C}[[x]]$.}
The inverse function theorem implies that 
 the inverse map $f^{-1}$ has similar properties.
 Let $r_0,\epsilon_0>0$ and
consider the set $B(r_0,\epsilon_0,{n_0})$ 
 defined by the inequality 
\begin{equation}\label{Eq:balleps0}
\max\left(\|f-\id\|_{r_0},\|f^{-1}-\id\|_{r_0}\right)\le \epsilon_0,
\end{equation}
where the supremum norm $\|\cdot\|_{r_0}$ is taken over 
the ball $\{x\in\C^d:|x|\le r_0 \}$.
Similar to the previous section we use the maximum 
component norm for vectors. 

For an integer $m\ge1 $ the Stirling--Newton symmetric interpolation scheme
provides an explicit expression for 
  symmetric interpolating vector field of order $2m$,
 \begin{equation}\label{Eq:Xm-symmetric}
X_{2m}(x) = \sum_{k=0}^{m-1} \frac{(-1)^k (k!)^2}{(2k+1)!} \cdot 
\frac{\Delta^{2k+1} (x_{-k-1}) + \Delta^{2k+1} (x_{-k})}{2},
\end{equation}
where $x_k=f^k(x)$ are points on the orbit of $x$
and the finite differences are defined recursively: $\Delta ^{j+1}(x)=\Delta ^{j}(f(x))-\Delta ^{j}(x)$ and $\Delta^1(x)=f(x)-x$.

\begin{thm}[Interpolating vector fields near the fixed point]\label{Thm:nearfixedpoint}
Let $\epsilon_0,r_0>0$, $\epsilon_0\le r_0$, and let $n_0$ be an integer such that ${n_0}\ge 2$. 
There is a constant $r_1\in (0,r_0)$ and a function 
$m:(0,r_1)\to \N$ such that 
\[
m(r)=  \left(\frac{r_0}{r}\right)^{n_0-1}
\left(\frac{1}{(n_0-1)\sqrt2\,\mathrm e^2 }
\cdot\frac{r_0}{\epsilon_0} +O\left(r^{n_0-1}\right)
\right)
\]
and  for any map $f\in B(r_0,\epsilon_0,n_0)$ and  
 any $r\in(0,r_1)$ 
the approximation error is bounded by 
\[
\bigl\|\Phi_{X_{m(r)}}^1-f\bigr\|_{r}
\le 4 \, \epsilon_0 \exp\left(-2 \, m(r)\right).
\]
\end{thm}

The proof of the theorem is split into several steps. First we obtain
an estimate for escape times from a neighborhood
of the origin in $\C^d$.
These bounds are based on comparison with the ordinary differential equation
\(
\dot R=R^{n_0}
\)
that has an obvious  strictly increasing positive solution 
\[
R(t)=\left(1-(n_0-1)t\right)^{-1/(n_0-1)}
\]
defined for $t<1/(n_0-1)$. We will also need its inverse function
$T:\mathbb R_+\to \mathbb R$, 
\[
T(r)=\frac1{n_0-1}\left(1-r^{-(n_0-1)}\right).
\]

\begin{lemma}[Escape times]\label{Le:escape}
If a map $f\in B(r_0,\epsilon_0,n_0)$ and $k\in\N$, then 
$|f^{k}(x)|\le r_0$  provided   $|x|\le \alpha_k$
where
 \begin{equation}\label{Eq:rho_k}
 \alpha_k=r_0\left(1+\frac{\epsilon_0}{r_0}(n_0-1)k\right)^{-1/(n_0-1)}.
 \end{equation}
 \end{lemma}
\begin{proof}
The MMP implies that for any $r\in [0,r_0]$
\[
\max\left(\bigl\|f-\id\bigr\|_{r},\bigl\|f^{-1}-\id\bigr\|_{r}\right)\le  \epsilon_0 (r/r_0)^{n_0}. 
\]
Let $\epsilon_1=\epsilon_0/r_0^{n_0} $.
The iterates of the map $f$ and its inverse $f^{-1}$
are dominated by the solution 
$r(t)=R(\epsilon_1 t)$ of the differential equation
$\dot r=\epsilon_1r^{n_0}$ while the point remains inside the domain of the map. If $t_0=\epsilon_1^{-1}T(r_0)$, then $R(\epsilon_1 t_0)=r_0$. 
For any $x$ with $|x|\le \alpha_k:=R(\epsilon_1 (t_0-k))$ 
we get that $\bigl|f^j(x)\bigr|\le
R(\epsilon_1 (t_0-k+j))\le r_0$
 for all $|j|\le k$. Using the explicit expressions for the functions $R$ and $T$ 
 we get the explicit expression \eqref{Eq:rho_k} for $\alpha_k$.
\end{proof}

\begin{lemma}[Interpolating vector fields]
The  symmetric interpolating vector field $X_{2m}(x)$
defined in \eqref{Eq:Xm-symmetric} with  $m\in\mathbb N$ 
is analytic in the disk $|x|\le \alpha_{m-1}$
and 
\begin{equation}\label{Eq:Xm-bound}
    \|X_{2m}\|_{r}\le \sqrt{2\pi}\epsilon_0\left(\frac{r}{r_0}\right)^{n_0}
\end{equation}
provided
\begin{equation}\label{Eq:beta_m}
r\le  \beta_m=\frac{\alpha_{m-1}}{(\sqrt 2\mathrm e)^{1/(n_0-1)}}.    
\end{equation}

\end{lemma}
\begin{proof} 
The interpolating vector field $X_{2m}(x)$
is analytic in the disk $|x|\le \alpha_{m-1}$ as it 
is a linear combination of the points
\(f^{-m}(x),\ldots,x,\dots,f^m(x)\)
and $|f^k(x)|\le r_0$ for $|k|< m-1$.

In order to get the upper bound we  
expand expressions for the finite  differences:
\[
\frac{\Delta^{2k+1} (x_{-k-1}) +\Delta^{2k+1} (x_{-k})}{2}
=
 \sum_{i=-k}^{k} (-1)^{i+k} \binom{2k}{i+k} \, \frac{\Delta^1( x_{i-1})+\Delta^1( x_i)}{2}.
\]
Lemma~\ref{Le:escape} implies that 
 the first $k$ iterates (both forwards and backwards) of a point $x$ with $|x|\le \alpha_{k}$ remain inside the ball of radius $ r_0$. Consequently
 $|\Delta x_i|=|x_{i+1}-x_i|\le\epsilon_0 $ for $i=-k-1,\ldots,k$ and
\[
\left\|
\frac{\Delta^{2k+1} (x_{-k-1}) + \Delta^{2k+1} (x_{-k})}{2}
\right\|_{r}
\le
\sum_{i=-k}^{k}  \binom{2k}{i+k} \epsilon_0
=2^{2k}\epsilon_0.
\]
Since $\Delta^j(x_i)=\Delta^{j-1}(x_{i+1})-\Delta^{j-1}(x_{i})$ and each application of $\Delta$ increases valuation
of a series by $n_0-1$, we get that
$
\operatorname{val}(\Delta^{2k+1} (x_j))\ge (2k+1)(n_0-1)+1.
$
Then the MMP implies that for $r\le \alpha_{k}$
\[
\left\|
\frac{\Delta^{2k+1} (x_{-k-1}) + \Delta^{2k+1}( x_{-k})}{2}
\right\|_{r}
\le
2^{2k}\epsilon_0\left(\frac{r}{\alpha_{k}}\right)^{(2k+1)(n_0-1)+1}
.
\]
Using the Stirling formula 
$
\sqrt{2\pi n}\left(\frac{n}{e}\right)^{n}
<
n!
<
\sqrt{2\pi n}\left(\frac{n}{e}\right)^{n} e^{\tfrac{1}{12n}}
$
and the classical bound $(1-1/n)^n=e^{n\log(1-1/n)}\le e^{-1-1/(2n)}$
we get for any $k \geq 1$
\[
\begin{split}
\frac{(k!)^2}{(2k+1)!}
&
\le
\frac{e^{1+1/(6k)}\sqrt{2\pi} }{\sqrt{\,2k+1\,}}
\left(\frac{k}{2k+1}\right)^{2k+1}
\\&
=
\frac{e^{1+1/(6k)}\sqrt{2\pi}}{\sqrt{2k+1}\,2^{2k+1}}
\left(1-\frac{1}{2k+1}\right)^{2k+1}
\le 
\frac{\sqrt{2\pi}}{\sqrt{2k+1}\,2^{2k+1}},
\end{split}
\]
Obviously this inequality also holds for $k=0$.
Using the last two bounds and  \eqref{Eq:Xm-symmetric}
we get that for $r\le \alpha_{m-1}$
\[
\begin{split}
\bigl\|X_{2m}\bigr\|_{r}&\le  
\sum_{k=0}^{m-1} 
\frac{\sqrt{2\pi}\epsilon_0}{2\sqrt{2k+1}}
\left(\frac{r}{\alpha_{k}}\right)^{(2k+1)(n_0-1)+1}
.
\end{split}
\]   
In order to estimate the ratio of two consecutive terms of the sum
we consider
\[
\begin{split}
\frac
{\left({r}/{\alpha_{k}}\right)^{(2k+1)(n_0-1)+1}}
{\left({r}/{\alpha_{k-1}}\right)^{(2k-1)(n_0-1)+1}}
&=\left(\frac{r}{\alpha_{k}}\right)^{2(n_0-1)}
\left(\frac{\alpha_{k-1}}{\alpha_{k}}\right)^{(2k-1)(n_0-1)+1}
\\
&
=
\left(\frac{r}{\alpha_{k}}\right)^{2(n_0-1)}
\left(
\frac{1+\frac{\epsilon_0}{r_0}(n_0-1)k
}{
1+\frac{\epsilon_0}{r_0}(n_0-1)(k-1)
}
\right)
^{(2k-1)+1/(n_0-1)}
\\
&=
\left(\frac{r}{\alpha_{k}}\right)^{2(n_0-1)}
\left(
1+
\frac{1
}{
\frac{1}{\frac{\epsilon_0}{r_0}(n_0-1)}+
k-1
}
\right)^{2k-1+1/(n_0-1)}
\end{split}
\]
Using the inequality $(1+t^{-1})^t\le \mathrm e$ for $t>0$
we conclude that the ratio of two consercutive terms in the sum does not exceed $1/2$ 
provided $(r/\alpha_k)^{2(n_0-1)}\le   1/(2\mathrm e^2)$.
  Consequently, if $r\le\beta_m$
 the sum is not larger than the first term doubled and we get
 the  bound~\eqref{Eq:Xm-bound}.
\end{proof}

\begin{lemma}[Time-one maps]
The time-one map of the symmetric interpolating vector field $X_{2m}$ 
defined in \eqref{Eq:Xm-symmetric} with  $m\in\mathbb N$  satisfies the inequality
\begin{equation}\label{Eq:PhiXm-f}
    \left\|\Phi_{X_{2m}}^1-f\right\|_r
    \le
        4\epsilon_0\left(\frac{\beta_m}{r_0}\right)^{n_0}
    \left(
\frac{r}{\gamma_m}    \right)^{2 m(n_0-1)+1}
\end{equation}
for  $0<r\le \gamma_m$ where  
\begin{equation}\label{Eq:gamma_m}
   \gamma_m=\beta_m\left(1+(n_0-1)\epsilon_2\beta_m^{(n_0-1)}
\right)^{-1/(n_0-1)}
\end{equation}
and $\epsilon_2=\sqrt{2\pi}\epsilon_0/r_0^{n_0} $.
\end{lemma}

\begin{proof}
The bound \eqref{Eq:Xm-bound} implies that  solutions of the differential equation $\dot x=X_{2m}(x)$
are dominated by the function 
$r(t)=R(\epsilon_2 t)$ that satisfies  the differential equation
$\dot r=\epsilon_2r^{n_0}$ while the trajectory remains inside the ball $|x|\le\beta_m$. If $t_0=\epsilon_2^{-1}T(\beta_m)$, then $R(\epsilon_2 t_0)=\beta_m$. 
For any $x_0$ with $|x_0|\le \gamma_m:=R(\epsilon_2 (t_0-1))$ 
we get that $\bigl|\Phi^1_{X_{2m}}(x_0)\bigr|\le
R(\epsilon_2 (t_0))=\beta_m$.
Using the explicit expressions for $R(t)$ and $T(r)$
we get the expression \eqref{Eq:gamma_m}:
\[
\gamma_m=\left(1-(n_0-1)\epsilon_2(t_0-1)
\right)^{-1/(n_0-1)}
=\left((n_0-1)\epsilon_2+\beta_m^{-(n_0-1)}
\right)^{-1/(n_0-1)}.
\]
Then the time-one map of $X_{2m}$  admits the following upper bound
\[
\left\|
\Phi_{X_{2m}}^1-\id
\right\|_{\gamma_m}
\le 
\|X_{2m}\|_{\beta_m}\le
\sqrt{2\pi}\epsilon_0\left(\frac{\beta_m}{r_0}\right)^{n_0}
\]
and
\[
\|\Phi_{X_{2m}}^1-f\|_{\gamma_m}\le \|\Phi_{X_{2m}}^1-\id\|_{\gamma_m}+\|\id-f\|_{\gamma_m}
\le 
4\epsilon_0\left(\frac{\beta_m}{r_0}\right)^{n_0}
.
\]
Since $\operatorname{val}(\Phi_{X_{2m}}^1-f) \geq 2m(n_0-1)+1$,
the MMP implies 
\eqref{Eq:PhiXm-f} for $r\le \gamma_m$.
\end{proof}

\begin{proof}[Proof of Theorem~\ref{Thm:nearfixedpoint}]
The theorem is proved by choosing $m$  that approximately  minimizes the error in \eqref{Eq:PhiXm-f} for a fixed $r$.
We note that $\beta_m,\gamma_m\in(0,r_0)$  are monotone 
decreasing with $m$ and converge to zero when $m\to\infty$. We define $m(r)$ to be the 
largest $m\in\N$ such that $(\gamma_m/r)^{n_0-1}\ge \mathrm e $. Then 
we get  
\[
\left\|\Phi_{X_{2m}}^1-f\right\|_r
    \le
       4 \epsilon_0
        \exp(-2 m(r)).
\]
We note that for $m\ge1$
\[
\alpha_m
= {r_0}
{\left(\frac{\epsilon_0}{r_0}(n_0-1)m\right)^{-1/(n_0-1)}}
\left(1+O(m^{-1})\right).
\]
Since $\beta_m=\frac{\alpha_{m-1}}{(\sqrt 2\mathrm e)^{1/(n_0-1)}}$ 
and $\gamma_m=\beta_m\left(1+O(\beta_m^{n_0-1})\right)$ we conclude that
\[
\gamma_m= r_0\left(\frac{\epsilon_0}{r_0}\sqrt2 \mathrm e(n_0-1)m\right)^{-1/(n_0-1)}
\left(1+O(m^{-1})\right).
\]
Consequently, the equation ${\gamma_m}=r\mathrm e^{1/(n_0-1)}$ has a solution
\[
m=
\left(\frac{r_0}{r}\right)^{n_0-1}
\left(\frac{1}{ \sqrt 2  \mathrm e^2(n_0-1)}\frac{r_0}{\epsilon_0}
+O(r^{n_0-1})\right)
\]
and $m\ge1$ provided $r\le r_1=\gamma_1\mathrm e^{-1/(n_0-1)}$.
The function $m(r)$ is equal to the 
integer part of the right hand side.
\end{proof}

\section{Dynamics near a resonant equilibrium\label{Se:stability}}

\subsection{Stability of a fully resonant elliptic fixed point}

In this section we discuss an application of the theory developed in the previous section to 
study  stability of an elliptic resonant fixed point.
Suppose that an analytic symplectic map $f$ in a $2d$ dimensional space has
an equilibrium at the origin. Let $A_0=f'(0)$. 
The fixed point is called {\em elliptic\/} if all eigenvalues of $A_0$
are non-real numbers on the unit circle. 
Since the map is  real-analytic, the eigenvalues form pairs of 
complex conjugate numbers. Denote by $\lambda_1,\dots,\lambda_d$ the eigenvalues 
with positive imaginary parts.
 The fixed point is called {\em fully resonant\/}
if $\lambda_1^n=\lambda_2^n=\ldots=\lambda_d^n=1$ for some $n\in \N$.
Consider the $d$-dimensional lattice $\mathcal L\subset\Z^d$ that consists of all
 $m=(m_1,\dots,m_d)\in\Z^d$ such that
\[
\lambda_1^{m_1}\ldots\lambda_d^{m_d}=1.
\]
Suppose that all eigenvalues are simple. Then the matrix $A_0$ is diagonalizable and
the normal form theory (for example by combining the  results of \cite{Bazzani87} and \cite{BenettinG94})
tell us that there is  a formal Hamiltonian function 
\[
h_Y
=
\sum_{k-l\in\mathcal L}c_{k,l} z_1^{k_1}\ldots z_d^{k_d}\bar z_1^{l_1}\dots \bar z_d^{l_d}
\]
where $z_j=p_j+i q_j$ for $j=1,\ldots,d$, 
and a formal canonical change of variables 
such that the Taylor series of 
\[
f^n=C\circ \Phi^n_{h_Y}\circ C^{-1}
\]
where $\Phi^n_{h_Y}$ is the time-$n$ formal flow of the Hamiltonian
vector field defined by  $h_Y$.

The formal series $h_Y$ is called the {\em  normal form Hamiltonian\/} for the map $f$.
The real symmetry implies that $c_{k,l}=\bar c_{l,k}$. 
The normal form Hamiltonian has the group of symmetries, a finite group isomorphic to the group
 of matrices $A_0^k$ with $k=0,\ldots,n-1$, as $h_Y=h_Y\circ A_0^k$. In this equality we have identified the
 matrix with the corresponding linear map.
 We say that the resonance is strong if the lattice $\mathcal L$ contains a non-zero integral
 vector with $|m_1|+\ldots +|m_d|\le 4$.
In the absence of strong resonances, the leading part of the normal form Hamiltonian takes the form
\[
h_Y
=
\sum_{j=1}^d c_{jj} (z_j\bar z_j)^2
+O_5
\]
where the coefficients $c_{jj}$ are real. If $c_{jj}$ are all positive (or all  negative),  the leading part of $h_Y$ is a strictly convex function.
In the original coordinates, the formal series $h_X=h_Y\circ C^{-1}$ starts with terms of order four
and the Hamiltonian $h_X$ is convex in a neighbourhood of the origin.
In this case,  we can deduce effective stability for exponentially long times. 

Since $h_X$ starts with terms of order four, the corresponding time-$n$
map is $O(|x|^3)$ close to the identity. Consequently,
 the map $f^n$ satisfies assumptions of Theorem~\ref{Thm:nearfixedpoint} with $n_0=3$.
 Theorem~\ref{Thm:formal}
implies that the interpolating vector field $X_{m}$
of order $m$
coincides with the formal interpolating vector field in all orders less 
than  $N_m=m(n_0-1)+n_0=4m+1$.   
Truncating the Taylor series for $X_m$ at this order we get
a Hamiltonian vector field $\hat X_m$, the corresponding
Hamiltonian function coincides with $h_m$, the truncated normal form Hamiltonian
$h_{X}$ truncated at the order $N_m$. Then 
\[
|h_m\circ f^n(x)-h_m(x)|\le \left|h_m\circ \Phi^n_{\hat X_m}(x)-h_m(x)\right|+\left|h_m\circ \Phi^n_{\hat X_m}(x)-h_m\circ f^n(x)\right|.
\]
The first term vanishes as the Hamiltonian function is constant along the
trajectories of its Hamiltonian vector field. We conclude that 
\[
|h_m\circ f^n(x)-h_m(x)|\le \|h_m'\|_{r}\left\| \Phi^n_{\hat X_m}- f^n\right\|_r.
\]
Let us consider the set $\{x:h_m(x)\le E\}$ and denote 
by $D_m(E)$ the connected component that contains the origin.
Let \[E_m(r)=\sup_{E}\bigl\{E: x\in D_m(E) \implies |x|\le r \bigr\}\]
and  $D_{m,r}= D_m(E_m(r))$. 
The analysis of the leading terms
of $h_m$ shows that  $D_{m,r}$ contains a ball of radius $\sim r$ centered at the origin,
while  the definition of $D_{m,r}$ implies that it
is contained in the larger ball of radius $r$.

Finally, Theorem~\ref{Thm:nearfixedpoint} implies that a trajectory of $f^n$ that starts in $D_{m_r,r/2}$ needs  at least $\exp(-C/r^2)$ iterates to increase the energy and leave $D_{m_r,r}$.
In the $2d$-dimensional case with $d=1$ (e.g. in the Hénon map) the stability may be permanent
as the invariant KAM curves separate the phase space.
We note that convexity of level sets is sufficient but not necessary condition for stability. For example,  near the resonant point of
the Hénon map the function $h_m$ is not convex but nevertheless 
the corresponding  level lines are closed and prevent trajectories
from escaping (see Figure~\ref{Fig:levelsH}).
If $d>1$ the permanent stability does not follow from the existence
of an adiabatic invariant and an elliptic fixed point can be unstable~\cite{FMS2020}.

\subsection{Hidden symmetries of  interpolating flows\label{Se:symmetries}}

Without loosing in generality, we may place the fixed point of a map at the origin, then $f(0)=0$.
The point is called fully resonant if the Jacobian $f'(0)$ is a root of unity. Then $f^n$,
the $n$-th iterate of the map, is tangent to identity:
\[
(f^n)'(0)=\bigl(f'(0)\bigr)^n=I.
\]

It is well known that there is a unique formal vector field $X$ such that 
the formal series for the time-$n$ flow defined by $X$ coincide
with $f^n$, i.e. 
\(\Phi^n_X = f^n.\)
For example, this claim can be deduced from Theorem~\ref{Thm:formal}.
If $f$ is symplectic, then $X$ is Hamiltonian and obviously 
\(
h_X\circ f^n= h_X\circ \Phi^n_X=h_X.
\)
We show that the vector field $X$ and the Hamiltonian $h_X$ are
not only $f^n$-invariant but also $f$-invariant.

We called this symmetry hidden to stress the fact that the vector 
field computed using $f^n$, the $n$-th iterate of the map $f$, is invariant under the original map $f$.
We note that an interpolating vector field $X_m$ of finite order 
share a common jet with the formal interpolator. Therefore the following lemma
implies that $X_m$ possesses the symmetry up to an error term.

\begin{lemma}[symmetries of the formal interpolating vector field]
\label{Le:symmetry}
Let $f$ be a formal diffeomorphism with $f(0)=0$ such that $f^n $ is tangent to the identity 
for some $n\in\N$,  and let $X$ be the formal
vector field such that $f^n=\Phi^n_X$. Then the vector field $X$ is $f$-invariant, i.e.,
\begin{equation}\label{Eq:fsymm}
X\circ f = D(f) X,
\end{equation}
where $D(f)$ is the Jacobian matrix of $f$.
Moreover, if $f$ is symplectic, then $X$ is Hamiltonian with a formal Hamiltonian $h_X$
and 
\begin{equation}\label{Eq:h-f-invar}
h_X\circ f=h_X.
\end{equation}
\end{lemma}

\begin{proof}
In order to check the symmetry, we note that 
the normal form theory states that there is a formal change of coordinates $C$ with $C(0)=0$ 
that transforms the map $f$ to its normal form $g$, i.e., $C \circ f = g \circ C$ and 
$g$ has the symmetry
\begin{equation}\label{Eq:gL0}
g \circ L_0 = L_0 \circ g
\end{equation}
where $L_0$ is the	linear map defined by $g'(0)$, the Jacobian of $g$ at the origin.
Since $L_0^{-1}\circ g$ is tangent-to-identity,
there is a unique formal vector field $Y$ such that
\[
\Phi_Y^1 = L_0^{-1} \circ g.
\]
Substituting this equality into \eqref{Eq:gL0} we  get 
\[
L_0\circ \Phi^1_Y=\Phi^1_Y\circ L_0
\]
and consequently
\[
L_0\circ Y=Y\circ L_0.
\]
The last conclusion may be proved using that discrete averaging 
preserves  linear symmetries. Alternatively, the normal form theory
 can be used as the formal normal form $g$ contains  
 resonant terms only and the symmetry of $Y$ follows from the symmetry of monomials in the corresponding series. 
 
The symmetry of $Y$ implies that for any $t$
\[
\Phi_Y^t \circ  L_0=L_0 \circ  \Phi_Y^t.
\]
Consequently
\(
\Phi_Y^t\circ g = \Phi_Y^t\circ \Phi^1_Y\circ L_0=L_0\circ \Phi_Y^1\circ \Phi^t_Y=g\circ \Phi_Y^t.
\)
Therefore
\[
\Phi_Y^t\circ g =g\circ \Phi_Y^t.
\]
Obviously $g^n$ is tangent-to-identity.
Since $L_0$ and $g$ commute
\[
\Phi_Y^n=(L_0^{-1}\circ g)^n=L_0^{-n}\circ g^n=g^n.
\]
Therefore $Y$ is the formal interpolating vector field for $g^n$.
Its uniqueness implies that
\[
\Phi_X^t=C^{-1}\circ \Phi_Y^t\circ C.
\]
Then we get,
\[
\Phi_X^t\circ f=C^{-1}\circ \Phi_Y^t\circ C\circ C^{-1}\circ g\circ C=
C^{-1}\circ g \circ  \Phi_Y^t\circ C=f\circ \Phi_X^t.
\]
We have proved that $f$ commutes with the formal flow: for all $t$
\[
\Phi_X^t \circ f=f\circ \Phi_X^t.
\]
Taking the derivative with respect to $t$ at $t=0$ we see that  the vector field $X$ is invariant under iterates of $f$ and \eqref{Eq:fsymm} holds. 
In the case of a symplectic $f$, the vector field $X$ is Hamiltonian and  it follows that 
\[
  h_X \circ f - h_X=\mathrm{const}.
\]
Finally, $f(0)=0$ implies that the constant is zero and \eqref{Eq:h-f-invar} follows immediately.
\end{proof}

Note that  a family of maps $f_p(x)=f(x,p)$ that unfolds a fully resonant fixed point
can be considered as the map $(x,p)\mapsto (f(x,p),p)$
 in the extended phase space. For example,
this argument justifies the application of Lemma~\ref{Le:symmetry}
and the theory of discrete averaging of Section~\ref{Se:fixedpoint} to the conservative  Hénon map in Section~\ref{Se:henon}.

\section{Conclusions}

In this paper we developed the theory of discrete averaging that 
uses weighted averages of the iterates of $f$ to construct  autonomous vector fields 
$X_m$ such that the corresponding time-one flows $\Phi^1_{X_m}$
approximate the map. The construction is explicit and can be used to analyze long term
stability for discrete time dynamical systems. Moreover, if necessary, it
 can be easily implemented in form of a numerical scheme to produce both numerical
 approximations and analytic expressions for vector fields and adiabatic invariants.

The discrete averaging provides an alternative to the traditional approach based 
on the combination of the classical averaging and normal form theory. 
In contrast to the traditional approach the discrete averaging does not
rely neither on the suspension procedure nor on coordinate changes.

We showed that the discrete averaging substantially simplifies
 analysis of the accuracy for averaging procedures. In the analytic case, we provided  explicit  and uniform bounds for  approximation errors.

\section*{Acknowledgements}
A.V. is supported by the Spanish grant PID2021-125535NB-I00 funded by 
MICIU/AEI/10.13039/501100011033 and by ERDF/EU.

\bibliographystyle{plain}

\end{document}